\documentclass[a4paper,reqno,11pt,oneside]{amsart}
\usepackage[all]{xy}
\usepackage{cm}
\numberwithin{equation}{section}

\title{A note on spherical functors}
\author{Ciaran Meachan}
\begin{document}

\maketitle
{\let\thefootnote\relax\footnotetext{Supported by the EPSRC Doctoral Prize Research Fellowship Grant no. EP/K503034/1.}}

\begin{abstract}
We provide a %short and simple 
new and very short proof of the fact that a spherical functor between certain triangulated categories induces an autoequivalence.
\end{abstract}

\section*{Introduction}
Let $\cD(X)$ denote the bounded derived category of coherent sheaves on a smooth projective variety $X$. If $Y$ is another smooth projective variety then any object $\cP\in\cD(X\times Y)$ gives rise to a Fourier--Mukai functor $F=\Phi_\cP:\cD(X)\to\cD(Y)$, and we refer to the object $\cP$ %which is unique up to isomorphism 
as the Fourier--Mukai kernel of $F$. Similarly, if $Z$ is a third smooth projective variety and $\cQ\in\cD(Y\times Z)$ is any object then the composition $\Phi_\cQ\circ\Phi_\cP$ is induced by the convolution $\cQ\ast\cP\coloneqq\pi_{13*}(\pi_{12}^*\cP\otimes\pi_{23}^*\cQ)\in\cD(X\times Z)$. 

Now, since such an $F=\Phi_\cP:\cD(X)\to\cD(Y)$ has a left adjoint $L=\Phi_{\cP_L}$ and a right adjoint $R=\Phi_{\cP_R}$, we can use the unit and counit of adjunction to define new kernels via the following triangles:
\[\cC\to\Delta_*\cO_X\xra{\eta}\cP_R\ast\cP\qquad\text{and}\qquad\cP\ast\cP_R\xra{\eps}\Delta_*\cO_Y\to\cT.\]
The induced functors $C=\Phi_\cC:\cD(X)\to\cD(X)$ and $T=\Phi_\cT:\cD(Y)\to\cD(Y)$ are called the cotwist and twist of $F$, respectively. 

In this brief note, we give a short and simple proof, relying solely on the structure of adjunctions, and a classical result found in \cite{johnstone2002sketches}, of the following theorem.

\begin{thm*}[\ref{twist-is-an-equivalence} and \ref{R=CL[1]}]
Let $F=\Phi_\cP:\cD(X)\to\cD(Y)$ be a Fourier--Mukai functor between the bounded derived categories of two smooth projective varieties $X$ and $Y$, with left adjoint $L=\Phi_{\cP_L}$ and right adjoint $R=\Phi_{\cP_R}$. Suppose that the cotwist $C=\Phi_\cC$ is an autoequivalence of $\cD(X)$ and $\cP_R\simeq \cC\ast\cP_L[1]$ is any isomorphism. Then the canonical map $\cP_R\to \cP_R\ast\cP\ast\cP_L\to \cC\ast\cP_L[1]$ is an isomorphism and the twist $T=\Phi_\cT$ is an autoequivalence of $\cD(Y)$.
\end{thm*}

The observation that spherical functors give rise to interesting autoequivalences is well documented. It is hard to underestimate the importance of the paper \cite{seidel2001braid}. Their ideas were further developed in \cite{rouquier2006categorification,anno2007spherical} and the foundations were finally completed in \cite{anno2013spherical}. Other notable works include \cite{kuznetsov2015fractional,addington2011new,segal2016all}. Our proof is different from all of these and so we feel it is worthy of mention.
\bigskip

\textbf{Acknowledgements}: I thank Michael Wemyss and Richard Thomas for their encouragement, support and expert advice, as well as Alexander Kuznetsov for helpful discussions. I also thank Arend Bayer, Andreas Hochenegger, Andreas Krug and David Ploog for their constructive comments on a previous version. I am very grateful to the referee for their suggestions on how to improve readability and clarity.

\section{Preliminaries}\label{section-preliminaries}

\begin{dfn}
If $F=\Phi_\cP:\cD(X)\to\cD(Y)$ is a Fourier--Mukai functor with left adjoint $L=\Phi_{\cP_L}$ and right adjoint $R=\Phi_{\cP_R}$ then we distinguish the units and counits with subscripts as follows: \[\Delta_*\cO_X\xra{\eta_R} \cP_R\ast\cP\qquad\cP\ast\cP_R\xra{\epsilon_R}\Delta_*\cO_Y\qquad\Delta_*\cO_Y\xra{\eta_L} \cP\ast\cP_L\qquad\cP_L\ast\cP\xra{\epsilon_L}\Delta_*\cO_X.\]
If an argument only deals with one adjoint pair then we shall drop the subscripts.
\end{dfn}

\begin{lem}
\label{unit-maps-to-counit}
Units and counits are exchanged under the adjunction isomorphisms:
\begin{align*}
\Hom(\Delta_*\cO_X,\cP_R\ast\cP)&\simeq\Hom(\cP_L\ast\cP,\Delta_*\cO_X)\;;\;\eta_R\mapsto\epsilon_L\\
\Hom(\cP\ast\cP_R,\Delta_*\cO_Y)&\simeq\Hom(\Delta_*\cO_Y,\cP\ast\cP_L)\;;\;\epsilon_R\mapsto\eta_L.
\end{align*}
\end{lem}

\begin{proof}
For the first one, recall from \cite[Chapter IV]{maclane1971categories} that an adjunction is a bijection which assigns arrows according to a specific recipe. In particular, we have:
\[\xymatrix{
\Delta_*\cO_X\ar[d]_-{\eta_R} &\ar@{}[d]|\mapsto&\cP_L\ast\cP \ar[rr]^-{\cP_L\ast\cP\ast\eta_R} \ar[drr]_-{\epsilon_L} && \cP_L\ast\cP\ast\cP_R\ast\cP\ar[d]^-{\epsilon\;\simeq\;\epsilon_L\circ (\cP_L\ast\epsilon_R\ast\cP)} \\ 
\cP_R\ast\cP &&&& \Delta_*\cO_X,
}\]
where $\eps$ is the counit associated to the adjoint pair $(\cP_L\ast\cP,\cP_R\ast\cP)$. Using naturality of counits, we can rewrite this universal arrow as $\eps\simeq\epsilon_L\circ (\cP_L\ast\epsilon_R\ast\cP)$. Finally, the composition of the arrows being $\eps_L$ follows from convolving the triangular identity $(\eps_R\ast\cP)\circ(\cP\ast\eta_R)\simeq\id_\cP$ on the left with $\cP_L$; see \cite[Theorem IV.1.1.(ii)]{maclane1971categories}. The second statement follows from a similar argument.
\end{proof}

\begin{dfn}
Let $F=\Phi_\cP:\cD(X)\to\cD(Y)$ be a Fourier--Mukai functor between the derived categories of two smooth projective varieties $X$ and $Y$ with left adjoint $L=\Phi_{\cP_L}$ and right adjoint $R=\Phi_{\cP_R}$. Using the units and counits above, we can define the \emph{twist} $T=\Phi_\cT$ and \emph{cotwist} $C=\Phi_\cC$ of $F$ by the following exact triangles:
\[\cP\ast\cP_R\xra{\epsilon_R}\Delta_*\cO_Y\xra{\alpha_R} \cT\xra{\beta_R}\cP\ast\cP_R[1]\quad\;\text{and}\;\quad \cC\xra{\delta_R}\Delta_*\cO_X\xra{\eta_R} \cP_R\ast\cP\xra{\gamma_R}\cC[1].\] 
Similarly, we define the \emph{dual twist} $T'=\Phi_{\cT'}$ and \emph{dual cotwist} $C'=\Phi_{\cC'}$ of $F$ by:
\[\cT'\xra{\delta_L}\Delta_*\cO_Y\xra{\eta_L} \cP\ast\cP_L\xra{\gamma_L}\cT'[1]\quad\;\text{and}\;\quad \cP_L\ast\cP\xra{\epsilon_L}\Delta_*\cO_X\xra{\alpha_L} \cC'\xra{\beta_L}\cP_L\ast\cP[1].\]
\end{dfn}

\begin{lem}%[{\cite[Remark 2.14]{kuznetsov2015fractional}}]
\label{left-adjoint}
$C'=\Phi_{\cC'}$ and $T'=\Phi_{\cT'}$ are left adjoint to $C=\Phi_\cC$ and $T=\Phi_\cT$, respectively.
\end{lem}

\begin{proof}
This can be found in \cite[Remark 2.10]{kuznetsov2015fractional}. For a direct argument, first take left adjoints of the triangle $\cC\to\Delta_*\cO_X\smash{\;\xra{\eta_R}\;}\cP_R\ast\cP$ defining the kernel of the cotwist $C=\Phi_\cC$ to get an exact triangle $\cP_L\ast\cP\to\Delta_*\cO_X\to \cC'$ and then use Lemma \ref{unit-maps-to-counit} to see that the %adjunction isomorphism 
$\Hom(\Delta_*\cO_X,\cP_R\ast\cP)\simeq\Hom(\cP_L\ast\cP,\Delta_*\cO_X)$ maps the unit $\eta_R$ to the counit $\epsilon_L$. This shows that $C'\dashv C$. Similarly, we can show $T'\dashv T$.
%Similarly, we can take left adjoints of $\cP\ast\cP_R\xra{\epsilon_R}\Delta_*\cO_Y\to \cT$ to get $\cT'\to\Delta_*\cO_Y\to \cP\ast\cP_L$ and use the other\footnote{For the fact that $\Hom(\Phi \cP,\Psi)\simeq \Hom(\Phi,\Psi \cP_L)$, we refer the reader to \cite[Section 3.1]{caldararu2010mukai}.} adjunction isomorphism to see that $\Hom(\cP\ast\cP_R,\Delta_*\cO_Y) \simeq \Hom(\Delta_*\cO_Y,\cP\ast\cP_L)\,;\epsilon_R \mapsto \eta_L$. Therefore, $T'\dashv T$ as well. 
\end{proof}

\begin{lem}\label{basic-identities}
Let $F=\Phi_\cP:\cD(X)\to\cD(Y)$ be a Fourier--Mukai functor with right adjoint $R=\Phi_{\cP_R}$. Then we have the following natural isomorphisms: 
\begin{align*}
\cT\ast\cP[-1]\xra{\beta_R\ast\cP[-1]}\cP\ast&\cP_R\ast\cP\xra{\cP\ast\gamma_R}\cP\ast\cC[1],\quad\text{and}\\
\cP_R\ast\cT[-1]\xra{\cP_R\ast\beta_R[-1]}\cP_R\ast&\cP\ast\cP_R\xra{\gamma_R\ast\cP_R} \cC\ast\cP_R[1].
\end{align*}
Similarly, if $F=\Phi_\cP:\cD(X)\to\cD(Y)$ be a Fourier--Mukai functor with left adjoint $L=\Phi_{\cP_L}$ then we have the following natural isomorphisms: 
\begin{align*}
\cP\ast\cC'[-1]\xra{\cP\ast\beta_L[-1]}\cP\ast&\cP_L\ast\cP\xra{\gamma_L\ast\cP} \cT'\ast\cP[1]\quad\trm{and}\\
\cC'\ast\cP_L[-1]\xra{\beta_L\ast\cP_L[-1]}\cP_L\ast&\cP\ast\cP_L\xra{\cP_L\ast\gamma_L} \cP_L\ast\cT'[1].
\end{align*}
\end{lem}

\begin{proof}
These identities are standard, and all stem from the triangular identities associated to the adjoint pairs. For example, \cite[Theorem IV.1.1(ii)]{maclane1971categories} tells us that the composition $\cP_R \smash{\;\xra{\eta_R\ast \cP_R}\;} \cP_R\ast\cP\ast\cP_R \smash{\;\xra{\cP_R\ast\epsilon_R}\;} \cP_R$ is the identity on $\cP_R$. This provides a splitting of $\cP_R\ast\cP\ast\cP_R$ and allows us to complete the following diagram: 
\[\xymatrix{&& \cP_R\ast\cT[-1] \ar[d]_(0.4){\cP_R\ast\beta_R[-1]} \ar[drr]^[@!155]{\sim} &&\\
\cP_R \ar[rr]^-{\eta_R\ast\cP_R} \ar@{=}[drr] && \cP_R\ast\cP\ast\cP_R \ar[rr]^(0.45){\gamma_R\ast\cP_R} \ar[d]^{\cP_R\ast\epsilon_R} && \cC\ast\cP_R[1]\\
&& \cP_R &&}\]
using the octahedral axiom to get a functorial isomorphism: 
\[\cP_R\ast\cT[-1]\xra{\cP_R\ast\beta_R[-1]}\cP_R\ast\cP\ast\cP_R\xra{\gamma_R\ast\cP_R} \cC\ast\cP_R[1].\qedhere\]
\end{proof}

The following result is the key technical lemma which will allow us to easily deduce that the twist associated to a spherical functor is an autoequivalence.

\begin{lem}\label{wee-beauty}
Let $F=\Phi_\cP:\cD(X)\to\cD(Y)$ be a Fourier--Mukai functor with a right adjoint $R=\Phi_{\cP_R}$. If there is any natural isomorphism (not necessarily the unit of adjunction) between $\cP_R\ast\cP$ and $\Delta_*\cO_X$ then the unit $\eta:\Delta_*\cO_X\smash{\;\xra\sim\;} \cP_R\ast\cP$ of adjunction is an isomorphism. That is, $F=\Phi_\cP$ is fully faithful.
\end{lem}

\begin{proof}
This statement is the dual of \cite[Lemma 1.1.1]{johnstone2002sketches} %and we include the proof for completeness. 
translated into the setting of Fourier--Mukai functors.
\end{proof}

\section{Spherical Functors}
\begin{dfn}
We say that a Fourier--Mukai functor $F=\Phi_\cP:\cD(X)\to\cD(Y)$ with left adjoint $L=\Phi_{\cP_L}$ and right adjoint $R=\Phi_{\cP_R}$ is \emph{spherical} if the cotwist $C=\Phi_\cC$ is an autoequivalence of $\cD(X)$ and the canonical map: 
\[(\gamma_R\ast\cP_L)\circ (\cP_R\ast\eta_L): \cP_R\to \cP_R\ast\cP\ast\cP_L\to \cC\ast\cP_L[1],\] 
is a functorial isomorphism. 
\end{dfn}

\begin{rmk}
Proposition \ref{R=CL[1]} shows that if $C=\Phi_\cC$ is an autoequivalence then \emph{any} isomorphism $\cP_R\simeq \cC\ast\cP_L[1]$ %implies the canonical map 
ensures that $(\gamma_R\ast\cP_L)\circ (\cP_R\ast\eta_L)%:\cP_R\to \cP_R\ast\cP\ast\cP_L\to \cC\ast\cP_L[1]
$ is an isomorphism.
\end{rmk}

%We are now ready to prove the main result of this article:

\begin{thm}\label{twist-is-an-equivalence}
Let $F=\Phi_\cP:\cD(X)\to\cD(Y)$ be a Fourier--Mukai functor between the bounded derived categories of two smooth projective varieties $X$ and $Y$, with left adjoint $L=\Phi_{\cP_L}$ and right adjoint $R=\Phi_{\cP_R}$. 
\begin{itemize}
\item[(i)] If the canonical map $(\gamma_R\ast\cP_L)\circ (\cP_R\ast\eta_L):\cP_R\to \cP_R\ast\cP\ast\cP_L\to \cC\ast\cP_L[1]$ is an isomorphism then the unit $\eta_T:\Delta_*\cO_Y\smash{\;\xra\sim\;} \cT\ast\cT'$ of adjunction is an isomorphism. 
\item[(ii)] If the canonical map $(\gamma_R\ast\cP_L)\circ (\cP_R\ast\eta_L):\cP_R\to \cP_R\ast\cP\ast\cP_L\to \cC\ast\cP_L[1]$ is an isomorphism and $C=\Phi_\cC$ is an autoequivalence of $\cD(X)$ then $T=\Phi_\cT$ is an autoequivalence of $\cD(Y)$.
\end{itemize}
\end{thm}

\begin{proof}
(i) We use the triangles $\cP\ast\cP_R\to\Delta_*\cO_Y\to \cT$ and $\cT'\to\Delta_*\cO_Y\to \cP\ast\cP_L$ to construct a commutative diagram:
\[\xymatrix{\cP\ast\cP_R \ar[rr]^-{\cP\ast\cP_R\ast\eta_L} \ar[d]_-{\epsilon_R} && \cP\ast\cP_R\ast\cP\ast\cP_L \ar[rr]^-{\cP\ast\cP_R\ast\gamma_L} \ar[d]^-{\epsilon_R\ast\cP\ast\cP_L} && \cP\ast\cP_R\ast\cT'[1] \ar[d]^-{\epsilon_R\ast\cT'[1]}\\ 
\Delta_*\cO_Y \ar[rr]^-{\eta_L} \ar[d]_-{\alpha_R} && \cP\ast\cP_L \ar[rr]^-{\gamma_L} \ar[d]^-{\alpha_R\ast\cP\ast\cP_L} && \cT'[1] \ar[d]^-{\alpha_R\ast\cT'[1]}\\ 
\cT \ar[rr]^-{\cT\ast\eta_L} && \cT\ast\cP\ast\cP_L[1] \ar[rr]^-{\cT\ast\gamma_L} && \cT\ast\cT'[1].}\]
If we consider the top right square of the previous diagram together with the commutative square:
\[\xymatrix{\cP\ast\cP_L\ar@{=}[d]\ar@{=}[r]&\cP\ast\cP_L\ar@{=}[d]\\
\cP\ast\cP_L\ar@{=}[r]&\cP\ast\cP_L,}\]
consisting of four copies of $\cP\ast\cP_L$ and all maps being the identity, then we can use the natural map $\cP\ast\eta_R\ast\cP_L:\cP\ast\cP_L\to \cP\ast\cP_R\ast\cP\ast\cP_L$ to form a commutative diagram:
\[\xymatrix@!0@C=4.5em@R=4.5em{& \cP\ast\cP_R\ast\cP\ast\cP_L\ar[rrr]^-{\cP\ast\cP_R\ast\gamma_L}\ar'[d][dd]^-{\epsilon_R\ast\cP\ast\cP_L} &&& \cP\ast\cP_R\ast\cT'[1]\ar[dd]^-{\epsilon_R\ast\cT'[1]}\\
\cP\ast\cP_L\ar@{=}[dd]\ar@{=}[rrr]\ar[ur]^{\cP\ast\eta_R\ast\cP_L} &&& \cP\ast\cP_L\ar@{=}[dd]\ar[ur]^(0.4){(\cP\ast\cP_R\ast\gamma_L)\circ (\cP\ast\eta_R\ast\cP_L)} &\\
& \cP\ast\cP_L\ar'[rr][rrr]^-{\gamma_L} &&& \cT'[1].\\
\cP\ast\cP_L\ar@{=}[rrr]\ar@{=}[ur] &&& \cP\ast\cP_L \ar[ur]_{\gamma_L}&}\]
Indeed, the left face is just the triangular identity convolved with $\cP_L$ on the right; the top and bottom faces are clearly commutative and the commutativity of the right face follows from the commutativity of the other faces of the cube.
%:
%\begin{equation}\label{commutative-cube-TT'}
%\[
%(\epsilon_R\ast\cT'[1])\circ (\cP\ast\cP_R\ast\gamma_L)\circ (\cP\ast\eta_R\ast\cP_L) \simeq \gamma_L\circ (\epsilon_R\ast\cP\ast\cP_L)\circ (\cP\ast\eta_R\ast\cP_L)\simeq \gamma_L,
%\]
%%\end{equation}
%where the first isomorphism uses the commutativity of the back face and the second isomorphism uses the triangular identity %composed 
%convolved with $\cP_L$ on the right again.

Applying the octahedral axiom to the top and right faces of this commutative diagram produces the 
following commutative diagram of triangles:
%commutative diagram in Figure \ref{fig:commdiag}.
\begin{figure}[htbp!]
\[\xymatrix@!0@C=2.3em@R=2.1em
{
 && \cP\ast\cP_R \ar[rrrrrr]^-{(\cP\ast\gamma_R\ast\cP_L)\circ (\cP\ast\cP_R\ast\eta_L)} &&&&&& \cP\ast\cC\ast\cP_L[1] \ar[rrrrrr] &&&&&& \cP\ast\cQ[1]\ar[dddd]\\\\
 \cP\ast\cP_R \ar[rrrrrr]^-{\cP\ast\cP_R\ast\eta_L}\ar[dddd]_-{\epsilon_R} \ar@{=}[uurr] &&&&&& \cP\ast\cP_R\ast\cP\ast\cP_L \ar'[dd][dddd]^-{\epsilon_R\ast\cP\ast\cP_L} \ar[uurr]^-{\cP\ast\gamma_R\ast\cP_L} \ar[rrrrrr]^-{\cP\ast\cP_R\ast\gamma_L} &&&&&& \cP\ast\cP_R\ast\cT'[1] \ar[uurr] \ar[dddd]^(0.3){\epsilon_R\ast\cT'[1]} &&\\\\
 &&&& \cP\ast\cP_L \ar[uurr]^-{\cP\ast\eta_R\ast\cP_L}\ar@{=}[dddd]\ar@{=}[rrrrrr] &&&&&& \cP\ast\cP_L \ar@{=}[dddd]\ar[uurr]^(0.4){(\cP\ast\cP_R\ast\gamma_L)\circ (\cP\ast\eta_R\ast\cP_L)} &&&& \Delta_*\cO_Y[1] \ar[dddd]\\\\
 \Delta_*\cO_Y \ar'[rrrr][rrrrrr]^-{\eta_L} \ar[dddd]_-{\alpha_R} &&&&&& \cP\ast\cP_L \ar'[dd][dddd]^-{\alpha_R\ast\cP\ast\cP_L} \ar'[rrrr][rrrrrr]^-{\gamma_L} &&&&&& \cT'[1] \ar[dddd]^(0.3){\alpha_R\ast\cT'[1]}\ar[uurr]^(0.6){\delta_L[1]} &&\\\\
 &&&& \cP\ast\cP_L \ar@{=}[uurr] \ar@{=}[rrrrrr] &&&&&& \cP\ast\cP_L \ar[uurr]^-{\gamma_L} &&&& \cT\ast\cT'[1]\\\\
 \cT \ar[rrrrrr]^-{\cT\ast\eta_L} &&&&&& \cT\ast\cP\ast\cP_L \ar[rrrrrr]^-{\cT\ast\gamma_L} &&&&&& \cT\ast\cT'[1]. \ar@{=}[uurr] &&
}\]
\caption{Diagram of functors %resolving 
associated to $\cT\ast\cT'[1]$.}\label{big-diagram-TT'}\label{fig:commdiag}
\end{figure}

The canonical map $(\gamma_R\ast\cP_L)\circ (\cP_R\ast\eta_L):\cP_R\to \cP_R\ast\cP\ast\cP_L\to \cC\ast\cP_L[1]$ is an isomorphism by assumption. Therefore, %composing 
convolving the canonical map with $\cP$ on the left, to get $(\cP\ast\gamma_R\ast\cP_L)\circ (\cP\ast\cP_R\ast\eta_L)$, must also be an isomorphism. This implies $\cP\ast\cQ[1]\simeq0$ which in turn provides a natural isomorphism $\Delta_*\cO_Y[1]\smash{\;\xra\sim\;} \cT\ast\cT'[1]$. By Lemma \ref{wee-beauty}, this implies that the unit $\eta_T:\Delta_*\cO_Y\smash{\;\xra\sim\;} \cT\ast\cT'$ of %the 
adjunction %$T'\dashv T$ 
is an isomorphism.% and hence $T':\cB\to\cB$ is fully faithful. 

(ii) Part (i) proves that $T'=\Phi_{\cT'}:\cD(Y)\to\cD(Y)$ is fully faithful and so it remains to show that $T'=\Phi_{\cT'}$ is an equivalence. By \cite[Lemma 1.50]{huybrechts2006fourier}, it is enough to show that $%\ker T=
\ker\Phi_\cT=0$. To see this, suppose that $\Phi_\cT(\cE)=0$ for some $\cE\in\cD(Y)$. Then, by Lemma \ref{basic-identities}, we have 
$%(\Phi_{\cC}\circ\Phi_{\cP_R})(\cE)=
\Phi_{\cC\ast\cP_R}(\cE)\simeq \Phi_{\cP_R\ast\cT}(\cE)[-2]=(\Phi_{\cP_R}(\Phi_\cT(\cE))[-2]=0$ which implies $\Phi_{\cP_R}(\cE)=0$ since the cotwist $C=\Phi_\cC$ is an autoequivalence by assumption. Now, the defining triangle $\Phi_\cP(\Phi_{\cP_R}(\cE))\to \cE\to \Phi_\cT(\cE)$ shows that $\cE=0$ and so $\ker \Phi_\cT=0$ as required. Finally, we know that $T=\Phi_\cT$ is right adjoint to $T'=\Phi_{\cT'}$ by Lemma \ref{left-adjoint} and so $T=\Phi_\cT$ must be an equivalence as well.
\end{proof}

\begin{cor}
The left (or right) adjoint of a spherical functor $F=\Phi_\cP$ with twist $T=\Phi_\cT$ and cotwist $C=\Phi_\cC$ is a spherical functor with twist $C^{-1}%=\Phi_{\cC'}
$ and cotwist $T^{-1}%=\Phi_{\cT'}
$.
\end{cor}

\begin{proof}
This follows from the fact that the units and counits are exchanged under adjunction; see Lemma \ref{unit-maps-to-counit}. 
\end{proof}

\section{Identifying adjoints by an autoequivalence}\label{appendix}
We work with the same notation that was introduced in Section \ref{section-preliminaries}. For details on adjunctions, we refer to \cite[\S IV.1--4]{maclane1971categories}.

%\begin{lem}\label{RF=CLF[1]}
%If $C=\Phi_\cC$ is an autoequivalence of $\cD(X)$ then the canonical maps:
%\begin{align*}
%&\psi:\cP_R\ast\cP\xra{\cP_R\ast\eta_L\ast\cP}\cP_R\ast\cP\ast\cP_L\ast\cP\xra{\gamma_R\ast\cP_L\ast\cP}\cC\ast\cP_L\ast\cP[1]\qquad\text{and}\\
%&\omega:\cP_L\ast\cT\ast\cP[-1]\xra{\cP_L\ast\beta_R\ast\cP[-1]}\cP_L\ast\cP\ast\cP_R\ast\cP\xra{\eps_L\ast\cP_R\ast\cP}\cP_R\ast\cP,
%\end{align*}
%are isomorphisms.
%\end{lem}%Can we weaken the hypothesis for the first/second iso to just $C'/C$ being fully faithful?

\begin{lem}\label{RF=CLF[1]}
If $C=\Phi_\cC$ is an autoequivalence of $\cD(X)$ then the canonical map:
\[\cP_R\ast\cP\xra{\cP_R\ast\eta_L\ast\cP}\cP_R\ast\cP\ast\cP_L\ast\cP\xra{\gamma_R\ast\cP_L\ast\cP}\cC\ast\cP_L\ast\cP[1],%\quad\trm{is an isomorphism.}
\]
is an isomorphism.
\end{lem}%Can we weaken the hypotheses to just $C'$ being fully faithful?

\begin{proof}
First, let us observe that we have a commutative diagram: 
\[\xymatrix@C=1.6em{\cP_R\ast\cP \ar[rrr]^-{\cP_R\ast\eta_L\ast\cP} \ar@{=}[drrr] &&& \cP_R\ast\cP\ast\cP_L\ast\cP \ar[d]^-{\cP_R\ast\cP\ast\epsilon_L} \ar[rrr]^-{\gamma_R\ast\cP_L\ast\cP} &&& \cC\ast\cP_L\ast\cP[1] \ar[d]^-{\cC\ast\epsilon_L[1]}\\ &&& \cP_R\ast\cP \ar[rrr]^-{\gamma_R} &&& \cC[1].}\]
Indeed, the left hand side is just the triangular identity convolved with $\cP_R$ on the left, and the square %on the right 
commutes since the arrows act on separate variables. 
That is, 
\begin{equation}\label{gamma_R}
\gamma_R\simeq (\cC\ast\epsilon_L[1])\circ (\gamma_R\ast\cP_L\ast\cP)\circ (\cP_R\ast\eta_L\ast\cP).
\end{equation}

Now, by Lemma \ref{unit-maps-to-counit} we know that $\eps_L$ is the left adjunct of $\eta_R$. Moreover, in the proof of Lemma \ref{left-adjoint}, we observed that the dual cotwist triangle is the left adjoint of the cotwist triangle. Thus, by comparing triangles, we see that $\alpha_L$ must be the left adjunct of $\delta_R$, or equivalently, $\delta_R$ is the right adjunct of $\alpha_L$. That is, we have:
\begin{equation}\label{delta-right-adjunct-of-alpha}
\begin{aligned}
\xymatrix{\Delta_*\cO_X\ar[d]_-{\alpha_L} &\ar@{}[d]|\mapsto&\cC \ar[rr]^-{\cC\ast\alpha_L} \ar[drr]_-{\delta_R} && \cC\ast\cC'\ar[d]^-{\eps_C}\\ \cC' &&&& \Delta_*\cO_X.}
\end{aligned}
\end{equation}

Combining the right hand side of \eqref{delta-right-adjunct-of-alpha} with \eqref{gamma_R}, we can observe that we have a commutative diagram of triangles: 
\[\xymatrix{\cP_R\ast\cP \ar[d]_-{(\gamma_R\ast\cP_L\ast\cP)\circ (\cP_R\ast\eta_L\ast\cP)} \ar[rr]^-{\gamma_R} && \cC[1] \ar[rr]^-{\delta_R[1]} \ar@{=}[d] && \Delta_*\cO_X[1] \\ \cC\ast\cP_L\ast\cP[1] \ar[rr]^-{\cC\ast\epsilon_L[1]} && \cC[1] \ar[rr]^-{\cC\ast\alpha_L[1]} && \cC\ast\cC'[1]\ar[u]_-{\eps_C[1]}^-\wr,}\]
where $\eps_C:\cC\ast\cC'\to\Delta_*\cO_X$ is an isomorphism since $C=\Phi_\cC$ is an autoequivalence by assumption. Since the second and third vertical arrows are isomorphisms, it follows that the first vertical arrow is also an isomorphism.
\end{proof}

\begin{prop}\label{R=CL[1]}
Suppose the cotwist $C=\Phi_\cC$ is an autoequivalence of $\cD(X)$. Then any isomorphism $\phi:\cP_R\smash{\;\xra\sim\;} \cC\ast\cP_L[1]$ implies the canonical map: 
\[\chi:\cP_R\xra{\cP_R\ast\eta_L}\cP_R\ast\cP\ast\cP_L\xra{\gamma_R\ast\cP_L}\cC\ast\cP_L[1],\]
is an isomorphism.
\end{prop}

\begin{proof}
If we consider the triangle $\cP_R\xra{\chi}\cC\ast\cP_L[1]\to\cQ$ then $\chi$ is an isomorphism if and only if $\cQ\simeq0$. To show that $\cQ\simeq0$ it is sufficient to prove that the induced Fourier--Mukai functor $\Phi_\cQ:\cD(Y)\to\cD(Y)$ is zero on a spanning class of $\cD(Y)$. To this end, we use the spanning class $\Omega=\im\Phi_\cP\cup\ker\Phi_{\cP_R}$ from \cite[\S2.4]{addington2011new}. Indeed, convolving the triangle defined by $\chi$ with $\cP$ gives: \[\cP_R\ast\cP\xra{\chi\ast\cP}\cC\ast\cP_L\ast\cP[1]\to\cQ\ast\cP,\] and Lemma \ref{RF=CLF[1]} tells us that $\chi\ast\cP$ is an isomorphism. That is, $\cQ\ast\cP\simeq0$ and so we see that $\Phi_\cQ$ is zero on $\im\Phi_\cP$. Next, we take an object $\cE\in\ker\Phi_{\cP_R}$ and evaluate the induced triangle of Fourier--Mukai functors on it to get: \[\Phi_{\cP_R}(\cE)=0\to\Phi_\cC\circ\Phi_{\cP_L}(\cE)[1]\to\Phi_\cQ(\cE).\] By assumption, we have some isomorphism $\phi:\cP_R\smash{\;\xra\sim\;} \cC\ast\cP_L[1]$ allowing us to conclude that $\Phi_\cC\circ\Phi_{\cP_L}(\cE)=0$ and hence $\Phi_\cQ(\cE)=0$. This shows that $\Phi_\cQ$ is also zero on $\ker\Phi_{\cP_R}$ and thus $\cQ\simeq0$, which completes the proof.
\end{proof}

\begin{rmks}
The hypotheses of the results in this section are stronger than necessary. Indeed, Lemma \ref{RF=CLF[1]} and Proposition \ref{R=CL[1]} only use the weaker statements that $\Phi_{\cC'}$ is fully faithful and $\ker\Phi_{\cP_R}=\ker\Phi_{\cP_L}$, respectively.
\end{rmks}

\bibliographystyle{alpha}
\bibliography{ref}

\end{document}